\documentclass[a4paper,12pt]{amsart}
\usepackage{mathrsfs}
\usepackage[leqno]{amsmath}
\usepackage{amstext, paralist, amsthm, amssymb, epsfig, paralist,dsfont}
\usepackage{txfonts}
\usepackage{amsfonts}
\usepackage{amsbsy}
\usepackage{amssymb}
\usepackage[latin1]{inputenc}
\usepackage[T1]{fontenc}
\usepackage{textcomp}
\usepackage{setspace}
\usepackage{amsbsy}
\usepackage{txfonts}
\usepackage{mathbbol}
\usepackage{bbm}
\usepackage[toc,page]{appendix}
\usepackage{amsrefs}

\theoremstyle{plain}

\newtheorem{Thm}[subsection]{Theorem}
\newtheorem{Prop}[subsection]{Proposition}
\newtheorem{Lem}[subsection]{Lemma}
\newtheorem{Cor}[subsection]{Corollary}

\newtheorem{Rem}[subsection]{Remark}

\newtheorem{de}[subsection]{Definition}

\newcommand{\comments}[1]{}

\newcommand{\C}{\mathbb{C}}

\newcommand{\N}{\mathbb{N}}
\newcommand{\R}{\mathbb{R}}

\newcommand{\La}{\Lambda}

\newcommand{\lra}{\longrightarrow}

\newcommand{\cE}{{\mathcal E}}

\newcommand{\cL}{{\mathcal L}}

\newcommand{\ces}{{\operatorname{ces}\nolimits}}
\newcommand{\cesp}{{\operatorname{ces(\emph{p})}\nolimits}}

\numberwithin{equation}{section}

\begin{document}
\title[Fr\'{e}chet and (LB) sequence spaces induced by dual Banach spaces]{Fr\'{e}chet and (LB) sequence spaces induced by dual Banach spaces of discrete Ces\`{a}ro spaces}
\author[J.\ Bonet and W.J.\ Ricker]{Jos\'{e} Bonet and Werner J.\ Ricker}

\vspace*{.3cm}

\noindent
\address{J. Bonet, Instituto Universitario de Matem\'{a}tica Pura y Aplicada IUMPA, Universitat Polit\`{e}cnica de Val\`{e}ncia, 46071 Valencia, Spain \newline
Email: jbonet@mat.upv.es}
\address{ W.J. Ricker: Math.-Geogr.\ Fakult\"{a}t, Katholische Universit\"{a}t Eichst\"{a}tt-Ingolstadt, 85072 Eichst\"{a}tt, Germany\newline
Email: werner.ricker@ku.de }

\maketitle\footnotetext{This article is accepted for publication in the Bulletin of the Belgian
Mathematical Society Simon Stevin.}

\begin{abstract}
The Fr\'{e}chet $(\mbox{resp.}, (\mbox{LB})\mbox{-})$ sequence  spaces $\ces(p+) := \bigcap_{r > p} \ces(r), 1 \leq p < \infty $ (resp.\
 $ \ces (p\mbox{-}) := \bigcup_{ 1 < r < p} \ces (r),
  1 < p \leq \infty),$ are known to be very different to the classical sequence
spaces $ \ell_ {p+} $ (resp.,  $ \ell_{p_{\mbox{-}}}).$ Both of these classes of  non-normable spaces $ \ces (p+),  \ces (p\mbox{-})$
are defined via the family of reflexive Banach sequence spaces
$ \ces (p), 1 < p < \infty .$ The \textit{dual}\/ Banach spaces $ d (q), 1 <  q < \infty ,$ of the discrete Ces\`{a}ro spaces $ \ces (p), 1 < p < \infty , $
 were studied by  G.\ Bennett, A.\ Jagers and others.
Our aim is to investigate in detail the corresponding sequence  spaces $ d (p+) $ and $ d (p\mbox{-}),$ which have not been considered before.
Some of their properties have similarities with those of  $ \ces (p+), \ces (p\mbox{-})$ but, they also exhibit differences.
For instance, $ \ces (p+)$ is isomorphic to a power series Fr\'{e}chet space of  order 1    whereas $ d (p+) $ is isomorphic to such a space of infinite order.
Every space $ \ces (p+), \ces (p\mbox{-})  $ admits an absolute basis but, none of the spaces $ d (p+),  d (p\mbox{-})$ have any absolute basis.
\end{abstract}

\section{Introduction} \label{S1} Perhaps the most important class of Banach sequence spaces is $ \ell_p , 1 \leq p \leq \infty , $
where $ \ell_p $ is equipped with its usual norm $ \|\cdot \|_p.$ A classical inequality of Hardy, \cite[Theorem 326]{HLP}, states
for $ 1 < p < \infty $ and with $ \frac{1}{p} + \frac{1}{p'} = 1 $ that
$$ \textstyle
 \sum^\infty_{ n = 1} \left( \frac{1}{n} \sum^n_{ k = 1} | x _k | \right) ^p \leq ( p')^p \sum ^\infty _{n = 1} | x_n | ^p , \quad  x = ( x_n ) _n \in \ell_p ,
 $$
  In terms of the Ces\`{a}ro operator $ C : \C ^{\N} \lra \C^{\N} ,$ defined by
 \begin{equation} \label{1.1} \textstyle
  C (x) := ( x _1, \frac{x_1 + x_2}{2} , \ldots , \frac{x_1 + x_2 + \ldots + x_n}{n}, \ldots  ) , \quad x = (x_n ) _{n }\in \C ^{\N}.
  \end{equation}
  Hardy's inequality can be formulated, for $ 1 < p < \infty,$ as
  \begin{equation} \label{1.2}
   \| C (| x | ) \| _p \leq p' \| x \| _p , \quad x \in \ell_p  , \end{equation}
   where $ | x | := ( | x_n | )_n $ for  $ x \in \C^{ \N} .$ Since $ C : \C^{\N}  \lra \C^{\N} $ is a positive operator
   (i.e., $ C (x) \geq 0 ,$ meant in the coordinatewise sense, whenever $ x \geq 0 $ in $ \C^{\N} ) $  and $ | C (x) | \leq C ( | x | ), $ it follows from \eqref{1.2}
   that $ C : \ell_p \lra \ell_p $ is a  continuous linear operator for all $ 1 < p < \infty .$  G.\ Bennett  investigated, in great detail, the closely related
   spaces
   \begin{equation} \label{1.3}
   \cesp := \{ x \in \C^{\N} : C ( | x | ) \in \ell_p \}, \quad 1 < p < \infty ,
   \end{equation}
   which are reflexive  Banach spaces relative to the norm
   \begin{equation}   \label{1.4}
   \| x \| _{\cesp} := \| C ( | x | ) \| _p , \quad x \in \cesp,
   \end{equation}
   and satisfy $ \ell_p \subseteq \cesp ;$ see \cite{B}, as well as \cite{AM1}, \cite{AM2}, \cite{BCD}, \cite{CR1}, \cite{G-E}, \cite{LM}, \cite{MPS},
   and the references therein. The Banach spaces $ \cesp$ have the desirable property (as do the spaces $ \ell_p)$ that they are \textit{solid}\/ in $ \C^{\N} ,$
   that is, if $ x \in \cesp $ and $ y \in \C^{\N} $ satisfy $ | y | \leq | x |, $ then $ y \in \cesp.$

The \textit{dual}\/ Banach spaces $ ( \cesp)' $ of $ \ces(p),  1 < p < \infty ,$ are rather complicated, \cite{J}. A more transparent \textit{isomorphic}\/ identification of
   $  ( \cesp)'$ occurs in \cite[Corollary 12.17]{B}. Indeed, it is shown there that
   \begin{equation} \label{1.5} \textstyle
   d (p) := \{ x \in \ell_ \infty : \hat{x} := ( \sup_{ k \geq n } \, | x_k| )_n \in \ell_p \}, \quad 1 < p < \infty ,
   \end{equation}
is a Banach space for the norm
\begin{equation} \label{1.6}  \textstyle
 \| x \| _{ d (p)} := \| \hat{x}\, \|_p, \quad x \in d (p),
 \end{equation}
 which is isomorphic to $ ( \ces(p'))',$ denoted by $d(p)\simeq ( \ces(p'))'$, with the duality given by
 $$  \textstyle
 \langle u, x \rangle := \sum ^\infty _{ n = 1} u_n x_n , \quad u \in \ces (p'), \quad x \in d (p) .
 $$
 It is clear from \eqref{1.5} and \eqref{1.6} that $ d(p)$ is also \textit{solid}\/ in $ \C^{\N},$ for $ 1 < p < \infty .$
 The Banach spaces $ d (p), 1 < p < \infty ,$ although less prominent than the discrete Ces\`{a}ro spaces $ \cesp, 1 < p < \infty ,$ have received some
 attention; see, e.g., \cite{B}, \cite{BR1}, \cite{BCD}, \cite{G-E}, \cite{J}, \cite{LM}.

Non-normable sequence spaces $ X \subseteq \C^{\N} $ are also an important part of functional analysis; see, for example, \cite{Bi},
 \cite{BMS}, \cite{GS}, \cite{Ja}, \cite{K\"{o}},  \cite{MV}, \cite{W} and the references therein. The classical Fr\'{e}chet spaces $ \ell_{ p+} := \bigcap_{ p < q} \ell_ q,$
 for $  1 \leq p < \infty ,$ are well understood, \cite{Di}, \cite{MM}. Their analogues $ \ces (p+) := \bigcap_{ p < q} \ces(q)$ were  introduced and investigated in
 \cite{ABR1}. The (inductive limit) $ (LB)$\textit{-space}\/ $ \ell_{p_{\mbox{-}}} := \mbox{ind}_{1 < q < p}  \ell_q, $ for $ 1 < p \leq \infty,$ which is the \textit{strong dual}\/ space of $ \ell_{p'+}, $ is also
 well known, \cite{MM}; for a recent study  of certain aspects of this class of spaces, see \cite{ABR2}. The analogous class of (LB)-spaces $ \ces(p\mbox{-}) := \mbox{ind}_{1 < q < p}  \ces(q) $, again an inductive limit, 
 is also analyzed in \cite{ABR2}. The purpose of this note is to introduce and identify properties of the Fr\'{e}chet spaces $ d (p+), 1 \leq p < \infty ,$ and of the  (LB)-spaces
 $ d (p{\mbox{-}}), 1 < p \leq \infty .$ These spaces have not been considered before.

 In the following Section  \ref{S2} we collect various relevant aspects of the Banach spaces $ \ell_p, \cesp $ and $ d (p), $ for $ 1 < p < \infty,$ which are needed in the sequel.
 Various inclusions between these spaces are recorded as well as some of the properties of the Ces\`{a}ro operator when it acts between pairs of such Banach spaces.

 Section \ref{S3}  summarizes some known properties of the non-normable sequence spaces $ \ell_{p+}, \ces (p+)$ and $ \ell_{p_{\mbox{-}}}, \ces (p\mbox{-}).$ Several new facts
 (cf. parts (b) of Proposition \ref{P3.3}(i), (ii), parts (ii), (iv), (v) of Proposition \ref{P3.4} and Proposition \ref{P3.6}) concerning these spaces are also
 established.

 Section \ref{S4} is the main one of this note; it exposes properties of the (solid) Fr\'{e}chet spaces $ d (p+), 1 \leq p < \infty ,$ and  of the (solid)  (LB)-spaces
 $ d (p\mbox{-}), 1 < p \leq \infty ,$ and compares them with those of the corresponding spaces $ \ell_{p+}, \ces (p+) $ and $ \ell_{p_{\mbox{-}}}, \ces (p\mbox{-}). $

\section{Preliminaries}  \label{S2}

The non-normable sequence spaces  $ \ell_{ p+}, \ces(p+), d (p+) $ and $ \ell_{p_{\mbox{-}}}, \ces (p\mbox{-}), d (p\mbox{-}) $ are assembled from the  Banach spaces $ \ell_p, \ces (q), d (s) $
for $ 1 < p, q, s < \infty .$ So, we begin by collecting some facts about various inclusions amongst these Banach spaces. For each $ n \in \N ,$ let  $ e_n := ( \delta _{kn})_k . $

\begin{Prop} \label{P2.1} \begin{itemize} \item[{\rm (i)}]
Each of the Banach spaces $ \ell_p, \ces(p) $ and $ d (p),$ for $ 1 < p < \infty ,$ is separable, reflexive and the canonical vectors $ \{ e_n : n \in \N \} $ form an unconditional9
basis. Moreover, $ ( \ell_p ) ' \simeq \ell_{p'} $ (isometrically) and $  ( \ces(p) )' \simeq d (p') .$

\item[{\rm (ii)}] Let $ 1 < p, q < \infty .$
\begin{itemize}
\item[{\rm (a)}] The inclusion $ \ell_p \subseteq \ell_q $ is satisfied and continuous if and only if $ p \leq q .$ Moreover, the inclusion is never compact. If $ p < q, $  then
$ \ell _p \subsetneqq \ell_q.$

\item[{\rm (b)}] The inclusion $ \ces (p) \subseteq \ces (q) $ is satisfied and continuous if and only if $ p \leq q.$ Moreover, the inclusion is compact if and only if $ p < q ,$
in which case $ \ces(p) \subsetneqq \ces (q). $

\item[{\rm (c)}] The inclusion $ d (p) \subseteq d (q)$ is satisfied and continuous if and only if $ p \leq q .$ Moreover, the inclusion is compact if and only if $ p < q,  $
in which case $ d (p) \subsetneqq d (q).$

\item[{\rm (d)}] The inclusion $ \ell_p \subseteq \ces (q) $ is satisfied and continuous if and only if $ p \leq q,$ in which case  $ \ell_p \subsetneqq \ces (q). $
Moreover, the inclusion is compact if and only if $ p < q .$

\item[{\rm (e)}] The inclusion $ d (p) \subseteq \ell_q $ is satisfied and continuous if and only if $ p \leq q,$ in which case $ d (p) \subsetneqq \ell_q.$ Moreover, the inclusion
is compact if and only if $ p < q .$

\item[{\rm (f)}] The inclusion $ d (p) \subseteq \ces (q) $ is satisfied and continuous if and only if $ p \leq q ,$ in which case $ d (p) \subsetneqq \ces (q). $
Moreover, the inclusion is compact if and only if $ p < q.$
\end{itemize}
\end{itemize}
\end{Prop}

Let us indicate where the various parts of Proposition \ref{P2.1} occur in the literature.

(i) \  For the $  \ell_p$-spaces the claims are well known. According to  \cite[p.61]{B},  \cite[Proposition 2]{J},  the spaces $ \ces(p), 1 < p < \infty ,$ are reflexive.
It was noted in Section  \ref{S1} that $ d (p)$ is isomorphic to $ ( \ces (p'))'$ and hence, also the spaces $ d (p), 1 < p < \infty,$ are reflexive. For  the unconditionality of
$ \{ e_n : n \in \N \} $ as a basis for $ \ces(p)$ we refer to  \cite[Proposition 2.1]{CR1} and for $ d (p) $ to  \cite[Proposition 2.1]{BR1}. It is then clear that the spaces
$ \ces (p), d (p), $ for $ 1 < p < \infty ,$ are separable.

(ii)  \  (a) \  The claims in this setting are all well known.

(b) \ For the first statement, see  \cite[Proposition 3.2(iii)]{ABR3}, and for the statement about compactness we refer to \cite[Proposition 3.4(ii)]{ABR3}.
The discussion prior to Proposition 3.3 in \cite{ABR3} shows that  $ \ces(p) \subsetneqq \ces (q) $ whenever $ p < q .$

(c) \  The first statement occurs in \cite[Proposition 5.1(iii)]{BR1} and the statement concerning compactness can be found in \cite[Proposition 5.2(iii)]{BR1}.
Proposition 2.7(ii) of \cite{BR1} reveals that $ d (p) \subsetneqq d (q) $ whenever $ p < q .$

(d) \ For the first statement we refer to \cite[Proposition 3.2(ii)]{ABR3} and \cite[Remark 2.2(ii)]{CR1}. Concerning the compactness statement, see  \cite[Proposition 3.4(iii)]{ABR3}.

(e) \ For the first statement, see Propositions 2.7(v) and 5.1(ii) in \cite{BR1}, and for the statement concerning compactness  we refer to \cite[Proposition 5.2(ii)]{BR1}.

(f) \  Proposition 5.1(i) of \cite{BR1} shows that $ d (p) \subseteq  \ces (q),$ with continuity of the inclusion, if and only if $ p \leq q . $ Moreover, if $ p \leq q, $ then part  (e)
yields that $ d (p) \subsetneqq \ell_q .$ Since $ \ell_p \subseteq \ces (q) $ by part (d), it follows that also $ d (p) \subsetneqq \ces (q).$  For the statement about compactness
we refer to \cite[Proposition 5.2(i)]{BR1}.

It is also relevant to clarify the non-isomorphic nature between the various Banach spaces  $ \ell_p, \ces (q)$ and $ d (s), $ for $ 1 < p, q, s < \infty .$

\begin{Prop} \label{P2.2} Let $ 1 < p, q < \infty .$
\begin{itemize}
\item[{\rm (i)}] The Banach spaces $ \ell_p$ and $  \ell_q$ are not isomorphic whenever  $ p  \neq q.$

\item[{\rm (ii)}] The Banach spaces $ \ces (p)$ and $ \ces (q)$ are not isomorphic whenever $ p \neq q.$

\item[{\rm (iii)}] The Banach spaces $ d (p) $ and $ d (q) $ are not isomorphic whenever $ p  \neq q.$

\item[{\rm (iv)}] The Banach spaces $ d (p) $ and $  \ces (q) $ are not isomorphic whenever $ p \neq q .$

\item[{\rm (v)}] The Banach spaces $ \ces (p) $ and  $ \ell_q $ are not isomorphic.

\item[{\rm (vi)}]  The Banach spaces $ d (p)$ and $ \ell_q $ are not isomorphic.
\end{itemize}
\end{Prop}
For the statements (i) - (vi) in Proposition \ref{P2.2} we refer successively to  \cite[p.54 of Vol.I]{LZ}, to \cite[Proposition 3.3]{ABR3}, to \cite[Proposition 2.7(ii)]{BR1}, to
\cite[Proposition 2.9(ii)]{BR1}, to \cite[Proposition 15.13]{B}, and to  \cite[Proposition 2.7(iv)]{BR1}.

The Ces\`{a}ro operator $ C :\C^{\N}  \lra \C^{\N} $ is a topological isomorphism of the Fr\'{e}chet space $ \C^{\N} $ onto itself, where $ \C^{\N} $ is equipped with the
coordinatewise convergence topology. It was noted in Section \ref{S1} that $ C : \ell_p \lra \ell_p $ is continuous for $ 1 < p < \infty .$ The same is true for
$ C : d (p) \lra d (p), $ \cite[Proposition 3.2(i)]{BR1}, and for $  C : \ces (p) \lra \ces(p);$  see the proof of Theorem 5.1 in \cite{CR1}. In relation  to
the Ces\`{a}ro operator $ C ,$
 the Banach spaces $ \ces (p), 1 <  p < \infty ,$ have a remarkable property, \cite[Theorem 20.31]{B}. Namely, for $ x \in \C ^{\N}, $
\begin{equation}  \label{2.1}
C ( | x | ) \in \ces (p) \  \ i\!f \ and \  only \  i\!f  \ \ x \in \ces(p),
\end{equation}
which immediately implies that
\begin{equation} \label{2.2}
C^2 ( | x | ) \in \ces(p) \ \ i\!f \  and  \ only \  i\!f  \  \ C ( | x |)  \in \ces(p).
\end{equation}
In view of \eqref{1.4}, the criterion \eqref{2.1} can be equivalently formulated, for  $ 1 < p < \infty, $ as
\begin{equation} \label{2.3}
C^2 ( | x | ) \in \ell_p \ \  i\!f \ and \  only \  i\!f \ \  C ( | x | ) \in \ell_p .
\end{equation}
For each $ 1 < p < \infty ,$ it is also known, for $ x \in \C ^{\N},$ that
\begin{equation} \label{2.4}
 C^2 ( | x | ) \in d (p) \ \  i\!f \  and \  only \  i\!f \ \  C ( | x | ) \in d (p),
 \end{equation}
 \cite[Proposition 3.7]{BR1}. The criteria \eqref{2.2}, \eqref{2.3} and  \eqref{2.4} reveal a surprising feature of the Ces\`{a}ro operator $ C $ when it acts on one
 of the Banach spaces $ \ell_p , \ces(p) $ or $  d (p), $ for $ 1 < p< \infty.$ For $ \ces (p)$ this is perhaps understandable, in view of  \eqref{1.3}  and
 \eqref{1.4}. However, the Banach spaces $ \ell_p$  and also $ d (p), $ as defined by \eqref{1.5} and \eqref{1.6}, have apriori no connection to $ C .$
 Actually more is true.

 For each $ 1< p < \infty ,$ let respectively $ [C, \ell_p]_s,  [C, d (p)]_s $ and  $  [C, \ces(p)]_s $ denote the \textit{largest solid}\/ vector space
 $ X \subseteq \C^{\N} $ such that  respectively $ C (X) \subseteq \ell_p, C (X) \subseteq d (p) $ and $ C (X) \subseteq \ces (p) .$ Relevant here is
 that $ C : \ces(p) \lra \ell_p $ continuously, which is clear from \eqref{1.4} and the inequality $  | C (x) | \leq C ( | x | ) ,$ that $ C : \ell_p \lra
 d (p) $ continuously, \cite[Proposition 3.4]{BR1}, and that $  C : \ces(p) \lra d (p) $ continuously, \cite[Corollary 3.8]{BR1}, which shows, in turn,
 that $ \ces(p) \subseteq [C, \ell_p ]_s, $ that $ \ell_p \subseteq [C, d (p)]_s,$ and that even $ \ces (p) \subseteq [C, d (p)]_s.$ Of course, also
 $ \ces(p) \subseteq [C, \ces (p)]_s $ as $ C $ maps $ \ces (p) $ into itself. Actually  the following equalities are valid:
 \begin{equation} \label{2.5}
 [C, \ell_p]_s = [C, d (p)]_s = [C, \ces(p)] = \ces (p), \quad 1 < p < \infty  .
 \end{equation}
 Indeed, for the fact that $ [C,  \ell_p]_s = [C, \ces (p)]_s = \ces(p), $ see p.62 and Theorem 2.5 of \cite{CR1}, and for the equality
 $ [C, d (p)]_s = \ces (p) $ we refer to \cite[Corollary 3.9]{BR1}.

 We end this section with  some inequalities needed later.

 \begin{Lem} \label{L2.3} Let $ 1 < p < \infty $ and $ n \in \N $ be arbitrary.

 \begin{itemize}
 \item[{\rm (i)}] $ | x _n | \leq n \,  \| x \|_{\ces(p)}, \quad x \in \ces (p) .  $

 \item[{\rm (ii)}]  $ | x_n| \leq \| x \| _{ d (p)}, \quad x \in  d (p) . $
 \end{itemize}
 \end{Lem}

 \begin{proof}  \ (i) \ Given $ x \in \ces (p) $ note that
 $$ \textstyle
 | x_n | \leq \sum^n _{ k = 1}  \, | x_k | = n ( \frac{1}{n} \sum^n _{k = 1} \, | x_k | ) = n ( C ( | x | ))_n \leq n \,
 \| C ( | x | ) \| _p = n  \, \|  x \| _{\ces (p)}.
 $$
 (ii) \ Fix $ x \in d (p) .$ It is clear from the definition of $ \hat{x} $  (cf.\ \eqref{1.4}) that $ 0 \leq | x _n e_n | \leq
 | x | \leq | \hat{x} | ,$ from which it follows that $
 | x_n  | = \| x_n e_n \| _p \leq \| \hat{x}\|_p = \| x \| _{d (p)} .
 $ \end{proof}

 \section{The non-normable spaces $ \ell_ {p+}, \ces (p+) $  and $ \ell_{p_{\mbox{-}}}, \ces (p{\mbox{-}})$ }  \label{S3}

 In this section we collect some facts on certain non-normable sequence spaces in $ \C^{\N}.$ Given
 $ 1 \leq p < \infty ,$ consider any strictly decreasing sequence
 \begin{equation} \label{3.1}
 \{ p_k \} ^\infty _{ k = 1} \subseteq ( p, \infty ) \quad \mbox{with} \quad p_k \downarrow p .
 \end{equation}
 Then $ \ell_{p+} := \bigcap _{ q > p} \ell_q = \bigcap^\infty _{ k = 1}  \ell_{p_k} $ is a Fr\'{e}chet space
 relative to the increasing sequence of norms on $ \ell_{ p+} $ given by $ \| \cdot \| _{p_k}  :
 x \longmapsto \| x \| _{p_k} ,$ for  $ k \in \N,$  \cite[Ch.1, \S 2]{GS}, \cite[\S 30.8]{K\"{o}}. Since each Banach space
 $ \ell_q, 1 < q < \infty , $ is solid, it is clear that $ \ell_{p+} $ is also solid in $ \C^{\N}. $  Similarly,
 $ \ces (p+) := \bigcap _{q > p} \ces (q) = \bigcap^\infty _{ k = 1} \ces (p_k) $ is a (solid) Fr\'{e}chet space relative
 to the increasing sequence of norms on $ \ces (p+) $ given by
 $ \| \cdot \| _{ \ces (p_k)} : x \longmapsto \| x \| _{\ces (p_k)} , $ for $ k \in \N .$ Clearly
 $$
 \ell_p \subseteq \ell_{p+} \subseteq \C^{\N}, \quad 1 \leq p < \infty ,
 $$
 with continuous inclusions; here Proposition \ref{P2.1}(ii)(a) is relevant. Also
 $$
 \ces(p) \subseteq \ces (p+) \subseteq \C^{\N} , \quad 1 \leq p < \infty ,
 $$
 with continuous inclusions (cf.\ Proposition \ref{P2.1}(ii)(b)).

 \begin{Prop} \label{P3.1} Let $ 1 \leq  p < \infty .$ The canonical vectors $ \{ e_n : n \in \N \} $  form an unconditional basis
 for $ \ell_{ p+} .$  Moreover, $ \ell_{p+} $ is reflexive but not Montel. With continuous inclusions we have
 \begin{equation} \label{3.2}
  \ell_{ p+}  \subseteq \ell_{q+} , \quad 1 \leq p \leq q < \infty .
  \end{equation}
 For every distinct pair $ p, q \in [1,\infty) $ the Fr\'{e}chet spaces $ \ell_{p+} $ and $ \ell_{q+} $
 are not isomorphic. In particular, the containment \eqref{3.2}  is proper whenever $  p < q .$
 \end{Prop}

 Some comments on Proposition \ref{P3.1} are in order.  That $ \{ e_n : n \in \N \} $  is an unconditional basis
 is well known and is a consequence of the fact that this is the case for each Banach space
 $ \ell_q, 1 < q < \infty ; $ see Proposition \ref{P2.1}(i). The same is true for the reflexivity of $ \ell_{p+} $
 as each space  $ \ell_q $ is reflexive for  $ 1 < q < \infty , $ \cite[Proposition 25.15]{MV}. According to  \cite{D}
 the space $ \ell_{p+} $ is \textit{not Montel}. For the continuity of the inclusions in \eqref{3.2},  see
 \cite[Proposition 26(i)]{ABR5}. According to Proposition 3.3  of \cite{ABR1} the spaces $ \ell_{p+} $  and $ \ell_{q+} $
 are not isomorphic whenever $ p \neq q . $  Finally, if there exist $ p <  q $ such that $ \ell_{p+} = \ell_{q+}, $ then the continuity of the
 inclusion \eqref{3.2} and the open mapping theorem for Fr\'{e}chet spaces,
 \cite[Theorem 24.30]{MV},  would imply that $ \ell_{p+} $ and $ \ell_{q+}$
 are isomorphic, which is not the case.

 The properties of the Fr\'{e}chet spaces $ \ces (p+), $  which we now record,  are very different to those of $ \ell_{p+}.$ For the remainder of this note we
 define $ \alpha \in \C^{\N} $ by  $ \alpha := ( \log (n))_n.$ The definition of a \textit{power series}\/ sequence space
 of order 2  can be found in \cite[IV Section 29]{MV}, for example.
 The power series spaces of order $ p $ for any $ p \in [1, \infty] $ are defined similarly; see e.g.,  \cite{V}.

\begin{Prop}\label{P3.2} \begin{itemize} \item[{\rm (i)}] Let $ 1 \leq p < \infty .$ \textit{The canonical vectors}\/ $ \{ e_n : n \in \N \} $  \textit{form an unconditional basis
for}\/ $  \ces (p+).$ With continuous inclusions we have
\begin{equation} \label{3.3}
\ell_{p+}  \subseteq \ces (q+), \quad 1 \leq p \leq q < \infty ,
\end{equation}
and these containments are proper. Also, with continuous inclusions we have
\begin{equation}\label{3.4}
\ces (p+) \subseteq \ces (q+), \quad 1 \leq p \leq q < \infty ,
\end{equation}
and these containments are proper whenever $ p < q .$

The space $ \ell_{p+}$ is not isomorphic to $ \ces (q+) $ for any $   1 \leq    q < \infty .$

\item[{\rm (ii)}] Each space $ \ces (p+), $ for $  1 \leq p < \infty ,$ is a Fr\'{e}chet-Schwartz space which is isomorphic to the K\"{o}the echelon space of order one
\begin{equation} \label{3.5} \textstyle
\Lambda_{-1/p'} (\alpha) := \{ x \in \C^{\N} : \sum^\infty _{ n = 1} | x_n | n^t < \infty , \quad \forall \: t < ( - \frac{1}{p'}) \},
\end{equation}
 which is  a power series space
 of finite type $ - 1/p' $ and order 1. In particular, every space $ \ces (p+) $ is isomorphic to the fixed power series space
 \begin{equation} \label{3.6} \textstyle
  \Lambda ^1_0 ( \alpha ) := \{ x \in \C^{\N} : \sum^\infty _{ n = 1} | x _n | n ^{-1 / k} < \infty , \quad \forall  k \in \N \}
  \end{equation}
  and hence, $ \ces (p+) $ is not nuclear.
  \end{itemize}
  \end{Prop}

  Concerning Proposition \ref{P3.2} consider first part (i). That $ \{ e_n : n \in \N \} $ is an unconditional basis is established in
  \cite[Proposition 3.5(i)]{ABR1}. For the fact that $ \ell_{p+} $ is not isomorphic to any space $ \ces (q+), $ for $ 1 \leq q < \infty , $
  see \cite[Proposition 3.5(iii)]{ABR1}. The continuity of the inclusion $ \ell_{p+} \subset \ces (q+) $ in \eqref{3.3} is established
  in \cite[Proposition 26(ii)]{ABR5}. If $ \ell_{p+} = \ces (q+) $ for some $ 1 \leq p \leq q < \infty ,$ then the continuity of the inclusion
  \eqref{3.3} and the open mapping theorem would imply that $ \ell_{p+} $ and $ \ces (q+) $ are isomorphic, which is not so.
  For the continuity of the inclusion in \eqref{3.4},  see \cite[Proposition 26(iii)]{ABR5}. According to \cite[Remark 3.4]{ABR1} the inclusion
  in \eqref{3.4} is necessarily proper if $ p < q .$ Now consider part (ii). That $ \ces(p+)$ is isomorphic to the space
  $ \Lambda _{-1/p'} ( \alpha), $ as given in \eqref{3.5}, is the statement of Theorem 3.1 in \cite{ABR1} and that each one  of these spaces is
  isomorphic to $ \Lambda ^1 _0 ( \alpha) $ is precisely Corollary 3.2 of \cite{ABR1}. Consequently, $  \ces (p+)$ is necessarily a
  \textit{Fr\'{e}chet-Schwartz}\/ space but, it is  \textit{not}\/ nuclear; see \cite[Proposition 3.5(ii)]{ABR1}.

  The following result summarizes certain properties of the Ces\`{a}ro operator that will be needed in the sequel.

  \begin{Prop} \label{P3.3} \begin{itemize} \item[{\rm (i)}]
  \begin{itemize}
  \item[{\rm (a)}] For $ 1 \leq p \leq q <  \infty ,$ the Ces\`{a}ro operator $ C : \ell_{p+} \lra \ell_{q+} $ is continuous.

  \item[{\rm (b)}] Let $ 1 \leq p < \infty $ and $ x \in \C^{\N} .$ Then
  $$
  C^2 ( | x| ) \in \ell_{ p+} \ \mbox{  if \  and \  only \  if } \   C ( | x | ) \in \ell_{p+}.
  $$

  \item[{\rm (c)}] For each $ 1 \leq p < \infty $ we have that  $ [C, \ell_{p+}]_s = \ces ( p+). $
  \end{itemize}

  \item[{\rm (ii)}] \begin{itemize} \item[{\rm (a)}] For $ 1 \leq p \leq  q < \infty $ the Ces\`{a}ro operator $ C : \ces (p+) \lra \ces (q+) $
  is continuous.

  \item[{\rm (b)}] Let $ 1 \leq p < \infty $ and $ x \in \C^{\N} .$ Then
  $$
  C^2 ( | x | ) \in \ces (p+)  \ \mbox{ if \ and \ only if } \ C (| x | ) \in \ces (p+).
  $$
  \item[{\rm (c)}] For each $ 1 \leq p < \infty $ we have that $ [C, \ces (p+)]_s = \ces (p+). $
  \end{itemize}
  \end{itemize}
  \end{Prop}

  \begin{proof} (i) \ (a) This is Proposition 28(i) of \cite{ABR5}.

  (b) \ Fix $ 1 \leq p < \infty $  and let $ x \in \C^{\N}.$ If $ C ( | x | ) \in \ell_{p+},$ then part (a) ensures that also
  $ C^2 ( | x | ) \in \ell_{p+}.$

  Conversely, suppose that $ C^2 ( |  x| ) \in \ell_{p+}.$ Fix any $ q > p .$ Since $ C^2 ( | x |)  \in \ell_{ p+} \subseteq \ell_q ,$ it follows
  from \eqref{2.3} that $ C ( | x | ) \in \ell_q .$ But, $ q > p $ is arbitrary, and so $ C ( | x | ) \in \ell_{p+} . $

  (c) \  See \cite[Proposition 2.5]{ABR1}.

  (ii) \ (a) \ See Proposition 28(iii) of \cite{ABR5}.

  (b) \ The proof of (b) in part (i) can easily be adapted to apply to this case (by using part (a) of (ii) and \eqref{2.2} in place of
   \eqref{2.3}).

   (c) \ See \cite[Proposition 2.6]{ABR1}.
   \end{proof}

   We now turn our attention to the (LB)-spaces $ \ell_{p_{\mbox{-}}}$ and $ \ces (p\mbox{-}).$  Given $ 1 < p \leq \infty ,$ consider any strictly increasing sequence
   \begin{equation} \label{3.7}  \textstyle
   \{ p_k \} ^\infty _{ k = 1} \subseteq ( 1, p) \quad \mbox{ with } \quad p_k \uparrow p .
   \end{equation}
   Define the linear spaces
   \begin{equation} \label{3.8}  \textstyle
   \ell_{p_{\mbox{-}}} := \bigcup_{ 1 < q < p} \ell_q \quad \mbox{ and } \quad \ces (p\mbox{-}) := \bigcup_{ 1 < q < p } \ces (q),
   \end{equation}
   and equip them with the inductive limit topology. In both cases the union is strictly increasing; see parts (a), (b) of Proposition \ref{P2.1}(ii).
   Accordingly, $ \ell_{p_{\mbox{-}}} = \mbox{ind}_k \ell_{p_k} $ and $ \ces (p\mbox{-}) = \mbox{ind}_k \, \ces (p_k) $ are (LB)-spaces, that is, a countable inductive
   limit of Banach spaces,   \cite{Bi}, \cite{BMS},\cite[pp.290--291]{MV} . Since the Banach spaces  $ \ell_q, \ces (q), $ for $ 1 < q < \infty ,$ are solid,
   it is clear from \eqref{3.8}  that both $ \ell_{p_{\mbox{-}}} $ and $ \ces (p\mbox{-})$ are \textit{solid}\/ in  $ \C^{\N}. $

For the definition of the \textit{strong dual}\/ space $ X'_\beta $ of a locally convex Hausdorff space $ X $ we refer to \cite[p.269]{MV}, for example.
   In the event that $ X $ is a Fr\'{e}chet space, the space $ X'_\beta $ is a  complete (DF)-space, \cite[Proposition 25.7]{MV}. Moreover, a Fr\'{e}chet space $ X$
   is reflexive if and only if $ X'_\beta $ is reflexive,  \cite[Corollary 25.11]{MV}. A reflexive Fr\'{e}chet space $X$ is  Montel if and only if
   $ X'_\beta $ is Montel, \cite[Proposition 24.25]{MV}. Recall that a locally convex inductive limit is said to be {\em regular} if each bounded set is contained and bounded in some step.

\begin{Prop} \label{P3.4} \begin{itemize} \item[{\rm (i)}]  For each
   $ 1 <  p \leq \infty,$ the space $ \ell_{p_{\mbox{-}}} $ is a  regular (LB)-space which is reflexive but not Montel.  It is a  (DF)-space which satisfies
   $ \ell_{p_{\mbox{-}}} \simeq ( \ell_{p'+})'_\beta $  and   $ ( \ell_{p_{\mbox{-}}})'_\beta \simeq \ell_{p'+} . $

   \item[{\rm (ii)}] For $ 1 < p \leq q \leq \infty $ the natural inclusions
   \begin{equation} \label{3.9} \textstyle
   \ell_{p-}  \subseteq \ell_{q-}
   \end{equation}

   are continuous. If $ p <  q, $ then $ \ell_{p_{\mbox{-}}} \subsetneqq \ell_{q_{\mbox{-}}}. $ Also,  the inclusions
   \begin{equation}  \label{3.10}
   \ell_{p_{\mbox{-}}}  \subseteq \ces (q\mbox{-}), \quad 1 < p \leq q \leq \infty ,
   \end{equation}
   are continuous and proper.

   \item[{\rm (iii)}] For $ 1 < p \leq q \leq \infty ,$ both of the Ces\`{a}ro operators $ C : \ell_{p_{\mbox{-}}} \lra \ell_{q_{\mbox{-}}} $ and
   $ C : \ces (p\mbox{-}) \lra \ell_{q_{\mbox{-}}} $ are continuous.

   \item[{\rm (iv)}] For $ 1 < p \leq \infty $ and $ x \in \C^{\N} $ it is the case that
   \begin{equation}  \label{3.11} \textstyle
   C^2 ( | x | ) \in \ell_{p_{\mbox{-}}} \ \mbox{ if \ and \ only \ if } \ C (| x | ) \in \ell_{ p_{\mbox{-}}} .
   \end{equation}

   \item[{\rm (v)}] The identity  $ [C, \ell_{p_{\mbox{-}}}]_s = \ces (p\mbox{-}) $ is valid for every $ 1 < p \leq \infty .$
   \end{itemize}
   \end{Prop}

   \begin{proof} \ (i) \ These facts are essentially known and follow from the general facts stated prior to the proposition;
   see also \cite{Bi}, \cite{Di}, \cite{MM}, for example.

   (ii) \ The continuity of the inclusion \eqref{3.9} occurs in \cite[Proposition 25(i)]{ABR2}. To see that the containment is proper
   when $ p< q, $ choose $ r \in (p, q) $ and note that $ \ell_{p_{\mbox{-}}} \subseteq \ell_r.$ According to part (a) of Proposition \ref{P2.1}(ii)
   there exists $ x \in \ell_q \bsl \ell_r .$ Then $ x \in \ell_{q_{\mbox{-}}} \bsl \ell_{p_{\mbox{-}}} $ because $ \ell_r \subseteq \ell_{q_{\mbox{-}}} . $

   For the continuity of the inclusion \eqref{3.10} see \cite[Proposition 25(ii)]{ABR2}. Suppose their exist $ p,q$ with
   $ 1 < p \leq q    \leq \infty $ such that $ \ell_{p_{\mbox{-}}} = \ces ( q\mbox{-}) .$  Since both $ \ell_{p_{\mbox{-}}}, \ces (q\mbox{-}) $ have a web and are
   ultra-bornological, \cite[Remark 24.36]{MV}, and $ \ell_{p_{\mbox{-}}} $ is continuously included in $ \ces (q\mbox{-}), $ it follows from the
   open mapping theorem, \cite[Theorem 24.30]{MV}, that $ \ell_{p_{\mbox{-}}} $ and $ \ces (q\mbox{-})$ are isomorphic. But, this is impossible as
   $ \ell_{p_{\mbox{-}}}$ is not Montel whereas $ \ces (p\mbox{-})$ is Montel, \cite[p.48]{ABR2}. Hence,  the inclusion \eqref{3.10} is always proper.

   (iii) \ See parts (i) and (iv) of Proposition 27 in \cite{ABR2}.

   (iv) \ Fix $ p \in (1, \infty ] $ and let $ x \in \C ^{\N} .$ If $ C ( | x | ) \in \ell_{p_{\mbox{-}}}, $ then also $
   C^2 ( | x | ) \in \ell_{p_{\mbox{-}}} $ by part (iii).  Conversely, suppose that  $ C^2 ( | x | ) \in \ell_{p_{\mbox{-}}} .$ According
   to \eqref{3.8} there exists $ q \in(1, p) $ such that $ C^2 ( | x | ) \in \ell_q $ and hence, by \eqref{2.3},
   $ C ( | x | ) \in \ell_q \subseteq \ell_{ p_{\mbox{-}}}.$

   (v) \ By part (iii) the operator $ C $ maps $ \ces (p\mbox{-}) $ into $ \ell_{p_{\mbox{-}}} ,$ which implies that $ \ces ( p\mbox{-}) \subseteq  [ C, \ell_{p_{\mbox{-}}}]_s.$
   Conversely, let $ X \subseteq  \C^{\N} $  be a solid subspace such that $ C (X) \subseteq \ell_{p_{\mbox{-}}} .$ Given $ x \in X ,$ also
   $ | x | \in X $ and so $ C ( | x | ) \in \ell_{p_{\mbox{-}}} .$
   Choose $ q \in (1,p) $ such that $ C ( | x | ) \in \ell_q   \subseteq \ces ( q) ;$ see \eqref{3.8}. Then \eqref{2.1} implies that $ x \in \ces (q) \subseteq \ces (p\mbox{-}). $
   Accordingly, $ X \subseteq \ces (p\mbox{-})$ which implies that $ [C, \ell_{p_{\mbox{-}}}]_s \subseteq \ces (p\mbox{-}). $
   \end{proof}

The (LB)-spaces  $ \ces (p\mbox{-})$ are rather different to the (LB)-spaces $ \ell_{p_{\mbox{-}}}.$ This is due to the fact, for $ 1 < p < q < \infty,$ that the natural inclusion
 map $ \ces(p) \subseteq \ces (q)$ is \textit{compact}\/ whereas the inclusion map $ \ell_p \subseteq \ell_q$ is \textit{not}\/ compact; see
 parts (a), (b) of Proposition  \ref{P2.1}(ii).
 For the definition of a (DFS)-space we refer to \cite[p.304 \& Proposition 25.20]{MV}; it is the strong dual of a Fr\'{e}chet-Schwartz space. In particular, a  (DFS)-space is complete
 and Montel, \cite[pp.61-62]{Bi}.

 \begin{Prop} \label{P3.5} \begin{itemize} \item[{\rm (i)}] With continuous inclusions  we have
 \begin{equation} \label{3.12}
  \ces (p\mbox{-}) \subseteq \ces (q\mbox{-}) \subseteq \C^{\N} , \quad 1 < p \leq q \leq \infty
  \end{equation}
  and these containments are proper whenever $\:  p < q  \:.$

  The space $ \ell_{p_{\mbox{-}}} $ is not isomorphic to $ \ces (q\mbox{-})$ for every pair $ 1 < p,  q \leq \infty .$

  \item[{\rm (ii)}] Each space $  \ces (p\mbox{-}),$ for $ 1 < p \leq \infty ,$ coincides algebraically and topologically with a countable inductive  limit $ k_1 ( \nu_p) $ of  weighted
  $ \ell_1$-spaces.This co-echelon space is isomorphic to the strong dual of the (power series) Fr\'{e}chet-Schwartz space $ \Lambda^\infty _{ 1 /p'} ( \alpha)  $
  of finite type $ 1 / p'$ and infinite order. In particular, $ \ces (p\mbox{-}) = ( \Lambda ^\infty _{ 1/p'} ( \alpha ) )'_\beta $ is a (DFS)-space but, it is not nuclear.
  Moreover, $ (\ces ( p\mbox{-}))'_\beta \simeq \Lambda ^\infty _{ 1/p'} ( \alpha).$ In addition, each space $ \ces (p\mbox{-}), $ for $ 1 < p \leq \infty, $ is isomorphic to the fixed
  (DFS)-space
  \begin{equation} \label{3.13}  \textstyle
   ( \La ^\infty _1 ( \alpha ))'_\beta = \ces ( \infty \mbox{-} ).
   \end{equation}

   \item[{\rm (iii)}] Whenever $ 1 < p \leq q \leq \infty ,$ the Ces\`{a}ro operator $ C : \ces (p\mbox{-}) \lra \ces (q\mbox{-})$ is continuous.

\item[{\rm (iv)}] For $ 1  < p \leq \infty $ and $ x \in \C^{\N} $ it is the case that
$$
C^2 ( | x | ) \in \ces (p\mbox{-}) \ \mbox{ if \ and \ only \ if }  \ C ( | x | ) \in \ces (p\mbox{-}).
$$
\item[{\rm (v)}] The identity $ [C, \ces (p\mbox{-})]_s  = \ces (p\mbox{-}) $ is valid for every $ 1 < p \leq \infty . $
\end{itemize}
\end{Prop}

Some comments relevant to Proposition \ref{P3.5} are as follows. For part (i), the continuity of the first inclusion in \eqref{3.12} occurs in \cite[Proposition 25(iii)]{ABR2}.
For the second inclusion in \eqref{3.12}, recall that the lcH-topology of the Fr\'{e}chet space $ \C^{\N} $ is given by the increasing sequence
of seminorms $ \{ q_m : m \in \N \}, $ where
$$  \textstyle
q_m (x) := \max_{ 1 \leq k \leq m}  | x_k | , \quad x =  ( x_n) _n \in \C^{\N} .
$$
 Given a fixed $    m  \in\N $ it follows from Lemma \ref{L2.3}(i), for each $ 1 < r < \infty ,$ that
$$
q_m (x) \leq m \,  \| x \| _{\ces(r)} ,  \quad x \in \ces (r).
$$
Accordingly, the inclusion $ \ces (r) \subseteq \C^{\N} $ is continuous for each
$ r \in (1, \infty )$,  which implies that also the inclusion $ \ces (q\mbox{-}) \subseteq \C^{\N} $ is continuous for each $ 1 < q \leq \infty , $
\cite[Proposition 24.7]{MV}. That the containment \eqref{3.12} is \textit{proper}\/ whenever $ p< q $ can be argued as in the proof
of part (ii) in Proposition \ref{P3.4} by replacing the use of Proposition \ref{P2.1}(ii)(a) there with Proposition \ref{P2.1}(ii)(b).
Corollary 4 of \cite{ABR2} shows that $ \ell_{p_{\mbox{-}}} $ is not isomorphic to $ \ces (q\mbox{-})$ for every $ 1 < p, q \leq \infty .$ All of the
assertions in part (ii) are proved on pp. 49-51 of \cite{ABR2}. For part (iii) we refer to \cite[Proposition 27(iii)]{ABR2}. The statement in part (iv)
follows directly from Proposition 1(i) of \cite{ABR2}. Finally the identity in (v) is Proposition 1(iv) of \cite{ABR2}.

The canonical vectors $ \{ e_n : n \in \N \} $  form an unconditional basis in the Fr\'{e}chet spaces $ \ces (p+), 1 \leq p < \infty $
(cf.\ Proposition \ref{P3.2}(i)), and a Schauder basis in each (DFS)-space $ \ces (p\mbox{-}), 1 < p \leq \infty , $ \cite[Proposition 1]{ABR2}. Actually
more is true. We recall the notion of an absolute basis, \cite[p.314]{Ja}, \cite[p.341]{MV}. Given a Schauder basis
$ \{ u_n : n \in \N \} $ of a locally convex Hausdorff space $ X ,$ for each $ x \in X  $ there exists a unique element $ (x_n) _n \in \C ^{\N}$
such that $ x = \sum^\infty _{ n = 1} x_n u_n , $ with the series converging in $ X $.  If, for each continuous seminorm $ p $ on $ X$
there exist a continuous seminorm $ q $ on $ X $ and $ A  > 0 $ such
that
$$ \textstyle
\sum^\infty _{ n = 1} | x_n| p (u_n ) \leq A  q (x), \quad x \simeq ( x_n ) _n \in X ,
$$
then $ \{ u_n : n \in \N \} $ is called an \textit{absolute basis}\/ of $ X .$

\begin{Prop}  \label{P3.6}
The canonical vectors $ \{ e_n : n \in \N \} $ form an absolute basis in each space $ \ces (p+), 1 \leq p < \infty ,$ and in each space $ \ces (p\mbox{-}), 1 < p \leq \infty . $
\end{Prop}

\begin{proof} \   Fix $  1 \leq p < \infty $ and recall (by Proposition \ref{P3.2}(ii)) that $ \ces (p+)$ is isomorphic to the K\"{o}the echelon space
$ \Lambda_{-1/p'} (\alpha). $ The conclusion then follows from \cite[Theorem 14.7.8]{Ja} or \cite[Proposition 27.26]{MV}.

Now fix $ 1 < p \leq \infty .$ As mentioned in part (ii) of Proposition  \ref{P3.5}, $ \ces (p\mbox{-}) $ equals a co-echelon space of order 1
given by $ k_1 ( v  _p ) = \mbox{ind}_n  \ell_1 ( {\rm v }_n)$
for the decreasing sequence of weights $ \nu_p = \{ {\rm v}_n : n \in \N \} ,$ where $ {\rm v }_n ( k) := k^{-1/q_n} $ for $ k \in \N ,$ with $ q_n \downarrow p';$
see \cite[p.49]{ABR2}. According to \cite{BMS} the space $ k_1 ( \nu_p)$ is a K\"{o}the sequence space $ K_1 ( \overline{V} )$ for some (uncountable) family of weights
$ \overline{V}$ associated with $ \nu_p.$ Then Theorem 14.7.8 (and its proof) in \cite{Ja} show that $ \{ e_n : n \in\N \} $
is an absolute basis of $ \ces(p\mbox{-}). $
\end{proof}

\begin{Rem} \label{R3.7}  {\rm Every absolute basis is an unconditional basis, \cite[p.314]{Ja}. Hence,    the Schauder basis
$ \{ e_n : n \in \N \} $ of $ \ces (p\mbox{-}), 1 < p \leq \infty ,$ is actually an unconditional basis. }
\end{Rem}

\section{The non-normable spaces $ d (p+) $ and $ d (p\mbox{-})$} \label{S4}

Fix $ p \in[1, \infty ).$ For any decreasing sequence $ \{ p_k \} ^\infty _{ k = 1} $ satisfying \eqref{3.1} the Fr\'{e}chet space $ d (p+)
:= \bigcap_{q >  p} d (q) =  \bigcap ^\infty _{ k = 1} d (p_k) $ is  defined relative to the increasing sequence of norms on
$ d (p+)$ given  by $ \|  \cdot \| _{ d (p_k)}  : x \longmapsto \| x \| _{d ( p_k)} , $ for $ k \in \N ;$ see \eqref{1.6}. Similarly,
let $ p \in (1, \infty ].$  For any increasing sequence $ \{ p_k \} ^\infty _{ k = 1} $ satisfying \eqref{3.7}  we define the  (LB)-space
$ d (p\mbox{-}) := \bigcup_{1 < q < p} d (q) = \bigcup^\infty _{k = 1}  d (p_k)$, equipped with the inductive limit topology. Since the Banach spaces $ d (q),$ for $ 1 <  q < \infty ,$ are solid,
both $ d (p+)$ and $ d (p\mbox{-})$ are \textit{solid}\/ in $ \C^{\N} .$ The aim of this section is to identify properties of this new class of
spaces and to compare them with those of $ \ell_{p+} , \ces (p\mbox{+}) $ and $ \ell_{p_{\mbox{-}}} , \ces (p\mbox{-})$ presented in Section \ref{S3}. It turns out that the
spaces $ d (  p+), $ resp. $ d (p\mbox{-}),$ have certain similarities with $ \ces (p+), $  resp. $ \ces (p\mbox{-}),$ but there are also major differences.

We begin with two lemmas which record a few basic features of the spaces $ d (p+) $ and $ d (p\mbox{-}).$ Since $ d (p) \subseteq \C^{\N},$  with a
continuous inclusion, for each $ 1 < p < \infty , $ \cite[Proposition 2.7(vi)]{BR1}, it follows that
\begin{equation} \label{4.1}
d (p) \subseteq d (p+) \subseteq \C^{\N} , \quad  1 < p < \infty ,
\end{equation}
and $ d (1+) \subseteq \C ^{\N} ,$ all with continuous inclusions, as well as
\begin{equation}  \label{4.2}
d (p+) \subseteq d (q+) , \quad 1 \leq p \leq q < \infty ,
\end{equation}
where one needs to invoke part (c) of Proposition 2.1(ii) to establish the  continuity of the inclusion in \eqref{4.2}.

\begin{Lem}  \label{L4.1}  Let $ 1 \leq p < \infty . $
\begin{itemize}
\item[{\rm (i)}] The canonical vectors $ \{  e_n : n \in\N \}$  form  an unconditional basis in $ d (p+).$

\item[{\rm (ii)}] The Fr\'{e}chet space $ d (p+)$ is reflexive and separable.

\item[{\rm (iii)}] The inclusion in \eqref{4.2} is proper whenever $ 1 \leq p < q < \infty .$
\end{itemize}
\end{Lem}
\begin{proof}  \ (i) \ Fix $  p \in [1, \infty )$ and recall that $ d (p+) = \bigcap^\infty _{ k = 1}  d (p_k)$ with $ \{ e_n : n \in \N \} $
forming an unconditional basis in each Banach space $ d (p_k) ,$ for $ k \in \N;$ see Proposition 2.1(i). So, given
$ x \in d  (p+)$ and any permutation  $ \pi : \N \lra \N $ it is the case that  $ \lim_{N \to \infty} \| x - \sum^N_{ n = 1} x _{ \pi (n)} e _{\pi(n)} \| _{d (p_k)}  = 0 $
for each $  k \in \N,$ that is, $ \lim_{N \to \infty} \sum^N _{ n = 1} x_{\pi (n)} e_{\pi (n)}  = x $ in $ d (p+).  $ Hence, the
series $ \sum^\infty_{ n = 1} x_n e_n $ is unconditionally convergent to $ x $  in $ d (p+).$

 (ii) \ The reflexivity of the Banach spaces $ d (q), 1 < q < \infty,$ implies the reflexivity  of $ d (p+). $ The separability of $ d (p+) $ is clear from part (i).

 (iii) Suppose that $ 1 \leq p < q < \infty $ satisfy $ d (p+) = d (q+).$ For any fixed $ r \in (p,q) $ it would follow from the containments
 $ d ( p+) \subseteq d (r) \subseteq d (q) \subseteq d (q+) $
 that $ d (r) = d (q).$ This contradicts part (c) of Proposition \ref{P2.1}(ii).
 \end{proof}

 According to parts (d), (e) of Proposition \ref{P2.1}(ii) it is the case that
 \begin{equation}  \textstyle \label{4.3}
 d (p) \subseteq \ell_q \subseteq \ces (r), \quad  1 < p \leq q \leq r < \infty ,
 \end{equation}
 with continuous inclusions, which implies that
 \begin{equation} \textstyle \label{4.4}
 d (p+) \subseteq \ell_{q+} \subseteq \ces (r+), \quad 1 \leq p \leq q \leq r < \infty ,
 \end{equation}
 also with continuous inclusions. A similar argument to the proof of part (iii) of Lemma \ref{L4.1} shows that the containments in \eqref{4.4}
 are always \textit{proper}.

 \begin{Lem} \label{L4.2} \begin{itemize}
 \item[{\rm (i)}] Each  (LB)-space $ d (p\mbox{-}), $ for $ 1 < p \leq \infty ,$ is a (DFS)-space. In particular, it is a Montel space.

 \item[{\rm (ii)}] With continuous inclusions it is the case that
 \begin{equation} \label{4.5}
 d (p\mbox{-}) \subseteq d (q\mbox{-}) \subseteq \C^{\N} , \quad 1 < p \leq q \leq \infty .
 \end{equation}
  In addition,  if $ p <  q ,$ then $ d (p\mbox{-}) \subsetneqq d (q\mbox{-}). $
  \end{itemize}
  \end{Lem}

   \begin{proof}   \ (i) \ Since the natural inclusion $ d (q) \subseteq d (r)$ is \textit{compact}\/ whenever
  $ 1 < q < r < \infty $ (cf. Proposition \ref{2.1}(ii)(c)), the (LB)-space $ d (p\mbox{-})$ is a (DFS)-space,
  \cite[Proposition 25.20]{MV}.

  (ii) \ An argument similar to that used to establish the continuity  of the second inclusion in \eqref{3.12} also
  applies here (by using part (ii) of Lemma \ref{2.3} in place of part (i)) to show that the second inclusion
  in \eqref{4.5}  is continuous.
  Similarly, the continuity of the inclusion $ d (s) \subseteq d (t) $ for $ 1 < s \leq t < \infty $ (cf.\ Proposition \ref{P2.1}(ii)(c)),
  implies that the inclusion $ d (p\mbox{-}) \subseteq d (q\mbox{-})$ in \eqref{4.5} is  continuous; see, for example, \cite[Lemma 17(i)]{ABR2}.

  If $ 1 <  p <   q \leq \infty ,$ then $ d (p\mbox{-}) \subsetneqq d (q\mbox{-}).$ Indeed, fix $ r \in (p,q). $ By Proposition \ref{P2.1}(ii)(c) there exists $ x \in
  d (r) \bsl d (p). $ Then $ x \in d (q\mbox{-})  $ but, $ x \not\in d (p\mbox{-})$ as $ d (s) \subseteq d (p) $ for all $ 1 < s < p .$
  \end{proof}

  The continuity of the inclusions \eqref{4.3} imply that
  \begin{equation} \label{4.6}
  d (p\mbox{-}) \subseteq \ell_{q_ {\mbox{-}}} \subseteq \ces ( r\mbox{-}) , \quad 1   < p \leq q \leq r \leq \infty ,
  \end{equation}
  with continuous inclusions; see  Lemma 17(i) of \cite{ABR2}. A similar argument to that used in the proof of  Lemma \ref{L4.2}(ii)
  shows that the containments in \eqref{4.6} are always \textit{proper}.

  Similarly, the continuity of the inclusion $ d (r) \subseteq d (p) $ whenever $ 1 < r \leq p < \infty  $ (cf.\ Proposition \ref{2.1}(ii)(c)), together
  with \cite[Lemma 17(i)]{ABR2} applied when $ T $ is the inclusion map from $ d (r)$ into $ d (p\mbox{-}), $ shows that
  \begin{equation} \label{4.7}
  d (r) \subseteq d (p\mbox{-}), \quad 1 < r < p \leq \infty,
  \end{equation}
  with continuous inclusions. An analogous argument to that used in the proof of part (ii) in Lemma \ref{L4.2} shows that $ d (r) \subsetneqq d (p\mbox{-}). $

  The following descriptions of $ d (p+)$ and $ d (p\mbox{-})$ exhibit important features of these spaces.

  \begin{Prop} \label{P4.3}  \begin{itemize} \item[{\rm (i)}] For  $ 1< p \leq \infty $ the map
  $ \Phi _p : ( \ces (p'+))' \lra d (p\mbox{-}) $ given by
  $$ \textstyle
  \Phi_{p}  (f) :=  ( \langle e_n, f \rangle ) _n , \quad f \in ( \ces ( p' +))',
  $$
  is a linear bijection and   a topological isomorphism of the (DFS)-space
  $ ( \ces (p'+))'_\beta  $ onto the (DFS)-space $ d (p\mbox{-}) = \mbox{ind}_k d (p_k), $
  where $ \{ p_k \}  ^\infty _{ k = 1} $ satisfies \eqref{3.7}.

  \item[(ii)] For each $ 1 \leq p < \infty $ the map $ \Psi _p : ( \ces (p'\mbox{-}))' \lra d (p+)$ given by
  $$  \textstyle
  \Psi _p (g) := ( \langle e_n, g \rangle ) _n , \quad g \in (\ces (p'\mbox{-}))',
  $$
  is a linear bijection and is a topological isomorphism of the Fr\'{e}chet-Schwartz space $ ( \ces (p'\mbox{-}))'_\beta $ onto the
  Fr\'{e}chet-Schwartz space $ d (p+).$
  \end{itemize}
  \end{Prop}

  \begin{proof} \ (i) \ By Proposition \ref{P3.2}(ii) the space $ \ces (p'+) $ is Fr\'{e}chet-Schwartz and by Lemma \ref{L4.2}(i) the space
  $ d (p\mbox{-}) $ is a  (DFS)-space. That $ \Phi_p $ is a  bijection and topological isomorphism follows from  \cite[Proposition 4.6]{ABR1}.

  (ii) \ Proposition \ref{3.5}(ii) shows that $ \ces (p'\mbox{-})$ is a (DFS)-space with compact linking maps (cf.\ Proposition  \ref{P2.1}(ii)(b)).
  Using \cite[\S 22.7, Theorem (9)]{K\"{o}}, and \cite[Proposition 25.20]{MV} it is routine to adapt the proof of Proposition 4.6 in \cite{ABR1}
  to prove that $ \Psi_p $ is a bijection and topological isomorphism.
 \end{proof}
\begin{Rem} \label{R4.4} {\rm  Since each space $ d (p+), 1 \leq p < \infty ,$ and $ d (p\mbox{-}), 1 < p \leq \infty, $ is reflexive, it is isomorphic to its bidual. So,
Proposition \ref{P4.3} implies that
$$
( d (p+))'_\beta \simeq (( \ces ( p'\mbox{-})'_\beta ) '_\beta \simeq\ces (p'\mbox{-}), \quad 1 \leq p < \infty ,
$$
and that
$$
 ( d (p \mbox{-} ))'_\beta \simeq (( \ces (p' + )'_\beta ) '_\beta \simeq \ces (p'+ ), \quad 1 < p  \leq \infty .
 $$ }
 \end{Rem}
 \begin{Cor}   \label{C4.5} \begin{itemize} \item[{\rm (i)}] Each space
 $ d (p\mbox{-}),$ for $ 1 < p \leq \infty ,$ is a non-nuclear, (DFS)-space which is isomorphic to the strong dual $ ( \Lambda ^1_0 ( \alpha))'_\beta $ of the
 power series Fr\'{e}chet space $ \Lambda ^1 _0 (\alpha)$ of finite type {\rm 0} and order {\rm 1}.

 \item[{\rm (ii)}] Each space $ d (p+),$ for $ 1 \leq p < \infty ,$ is a non-nuclear, Fr\'{e}chet-Schwartz space which is isomorphic to the power series
 space $ \Lambda^ \infty _0 (\alpha) $ of finite type {\rm 0}  and order infinity.
 \end{itemize}
 \end{Cor}
 \begin{proof} \  In view of Propositions  \ref{P3.2}(ii), \ref{P3.5}(ii), and \ref{P4.3}, only the non-nuclearity of the spaces in parts (i) and (ii) needs
 to be addressed. According to Proposition \ref{P3.2}(ii) and Proposition \ref{P3.5}(ii) the spaces $ \ces (q+), 1 \leq   q < \infty ,$ and $ \ces (q\mbox{-}),
 1 < q   \leq \infty ,$ are all non-nuclear.  Then  \cite[p.78, Theorem$'$]{P1} implies  that the strong dual spaces
 $ d (p\mbox{-}) \simeq ( \ces (p'+))'_\beta ,$ for $ 1 < p \leq \infty,$ and $ d (p+) \cong ( \ces (p'\mbox{-}))'_\beta,$ for $ 1 \leq p < \infty, $ are also non-nuclear.
 \end{proof}

 Proposition \ref{P3.6} implies that the unconditional basis $ \{ e_n : n \in\N \} $  in $ \ces (p+), 1 \leq p < \infty,$ and
 in  $ \ces (p\mbox{-}), 1 < p \leq \infty$ (cf.\ Proposition \ref{P3.2}(i) and Remark \ref{R3.7}) is actually an absolute basis. This turns out \textit{not}\/ to be the case for the spaces
 $ d (p+) $ and $ d (p\mbox{-}).$ For the definition of the operator ideals $ \cL_p, 1 \leq p \leq \infty ,$ of all $p$-factorable operators we refer to
 \cite{P2}; see also \cite{Ho}, where the terminology of an $ \cL_p$-locally convex space is used.

 \begin{Thm} \label{T4.6} The canonical vectors $ \{ e_n : n \in \N \} $ form an unconditional basis for each Fr\'{e}chet-Schwartz space $ d (p+), 1 \leq p < \infty ,$
 and for each (DFS)-space $ d (p\mbox{-}), 1  < p \leq \infty .$ However, none of these spaces have any absolute basis.
 \end{Thm}

 \begin{proof} \
 According to Lemma \ref{L4.1}(i), $ \{ e_n : n \in \N \} $ is an unconditional basis in $ d (p+), $  for $ 1 \leq p < \infty .$

 Fix $ 1 < p \leq \infty $ and let $ x \in d (p\mbox{-}).$ Then $ x \in d (r) $ for some $ 1 < r < p .$ Given any permutation $ \pi : \N \lra \N ,$
 it follows from Proposition \ref{P2.1}(i) applied to $ d (r),$ that $ \lim_{N \to \infty} \| x - \sum^N_{ n = 1}  x_{\pi (n) } e _{ \pi (n)} \|  _{ d(r)} = 0 .  $
 Since $ d (r) \subseteq d (p\mbox{-}) $ continuously (cf.\ \eqref{4.7}), it follows that $ \lim_{ N \to \infty}  \sum^N_{ n = 1} x _{ \pi (n)}  e_{\pi (n)} = x $
 in $ d (p\mbox{-}),$ that is, $ x = \sum^\infty _{ n = 1} x _{\pi(n)} e_{ \pi (n)} $ with the series converging in $ d (p\mbox{-}).$ But, $ \pi $ is arbitrary
 and so the series $ \sum^\infty _{ n = 1} x_n e_n $ converges unconditionally to $ x $ in $ d (p\mbox{-}).$ Hence, $ \{ e_n : n  \in \N \} $ is an unconditional basis
 for $ d (p\mbox{-}). $

 Let $ 1 \leq p < \infty$ and suppose that $ d (p+)$ has some absolute basis. By Proposition  27.26  of \cite{MV}, $ d (p+)$ would be isomorphic to a  K\"{o}the echelon space of order 1,
 that is, $ d (p+)$ is an $ \cL_1$-space. On the other hand, Corollary \ref{C4.5}(ii) implies that $ d (p+)$ is also an $ \cL_\infty$-space. According
 to \cite[Theorem 29.7.6]{P2} the space $ d (p+)$ is then nuclear, which contradicts Corollary \ref{4.5}(ii). Hence, $ d (p+)$ has \textit{no}\/ absolute basis.

 Suppose, for fixed $ 1 < p \leq \infty ,$ that $ d (p\mbox{-})$ has an absolute basis. By Theorem 14.7.8 of \cite{Ja} the complete (DFS)-space $ d (p\mbox{-})$ is isomorphic
 to a K\"{o}the sequence space of order 1 and hence, is an $ \cL_1$-space. According to Proposition \ref{P4.3}(i) and Corollary \ref{4.5}(i)  the space $ d (p\mbox{-})$ is also
 an $ \cL_\infty$-space. Hence, $ d(p\mbox{-})$ is nuclear, \cite[Theorem 29.7.6]{P2}, which contradicts Corollary \ref{4.5}(i). So, $ d (p\mbox{-})$ has \textit{no}\/ absolute
 basis.
  \end{proof}
  \begin{Thm} \label{T4.7} \ {\rm (i)} \ For each pair $ 1 \leq p, q < \infty ,$ the Fr\'{e}chet space $ \ces (p+)$ is not
  isomorphic to the Fr\'{e}chet space $ d (q+). $

  {\rm (ii)} \ For each pair $ 1 < p, q \leq \infty,$ the (DFS)-space $ \ces (p\mbox{-})$ is not isomorphic to the (DFS)-space $ d (q\mbox{-}).$
  \end{Thm}

    \begin{proof}  \ (i) \ Fix $ 1 \leq p, q < \infty .$ Since isomorphisms between locally convex Hausdorff spaces map absolute bases to
    absolute bases, it follows from Proposition \ref{P3.6} and Theorem \ref{4.6} that $ \ces (p+) $ cannot be isomorphic to $ d (q+).$

    (ii) \ Fix $ 1 < p, q \leq \infty .$ If $ \ces (p\mbox{-})$ and $ d (q\mbox{-})$ are isomorphic, then also their strong dual spaces
    $ d (p'+)$ and $ \ces (q'+)$ are isomorphic. This contradicts part (i).
    \end{proof}

  \begin{Rem} \label{R4.8}  \ {\rm (i)  \  An alternate proof of the fact that no space $ d (p+), 1 \leq p < \infty ,$ has an absolute basis is possible. Since
  $ (d ( p+))'_\beta = \ces (p'\mbox{-}) $  \textit{has}\/ an absolute basis (cf.\ Proposition \ref{P3.6}), \textit{if}\/ also $ d (p+)$ had an absolute basis, then
  Theorem 21.10.6 of \cite{Ja} would imply that $ d (p+)$ is nuclear; contradiction to Corollary  \ref{C4.5}(ii).

  {\rm (ii)} \ A further difference between the unconditional basis $   \cE := \{ e_n : n \in \N \} $ when it is considered to belong to one of the spaces $ \ces (p+), \ces (p\mbox{-}), $
  in contrast to when it is considered in one of the spaces $ d (p+), d (p\mbox{-}), $ should be pointed out. For each $ q \in (1, \infty )$ it is known that there exist positive
  constants $ A_q, B_q $ such that
  \begin{equation} \label{4.8} \textstyle
 A_q \, n ^{-1/q'} \leq \| e_n \| _{\ces(q)} \leq  B_q \, n^{-1/ q'}, \quad n \in \N ,
  \end{equation}
  \cite[Lemma 4.7]{B}. Fix $ p \in [1,\infty )$ and consider $ \cE \subseteq \ces (p+).$ Given $ q > p $ it follows from \eqref{4.8} that
  $ \lim_{ n \to \infty} e_n = 0 $ in $ \ces (q).$ Since $ \ces (p+) = \bigcap _{ k \in \N} \ces (p_k) $ for any $ p_k \downarrow p,$ it follows that
  $ \lim_{ n \to \infty} e_n = 0 $ in $ \ces (p+).$ Similarly, consider $ \cE \subseteq \ces (p\mbox{-}) $ for any fixed $ 1 < p \leq \infty .$ Since
  $ \ces (p\mbox{-}) = \mbox{ind}_k \ces (p_k)$ for any $ 1 <  p_k \uparrow p $ is equipped with its inductive limit  topology, \cite[p.280]{MV}, each inclusion map
  $ \ces (p_k) \subseteq \ces (p\mbox{-}) $ for $ k \in \N $ is continuous. Fix any $ k_0 \in \N .$ It is clear from \eqref{4.8} that $ \lim_{ n \to \infty} e_n = 0$
  in the Banach space $ \ces (p_{k_0}). $ By the previous comment this implies that $ \lim_{ n \to \infty} e_n = 0  $ in $ \ces (p\mbox{-}).$

  On the other hand, for each $ 1 < p < \infty ,$ it is straight-forward to check that
  \begin{equation} \label{4.9}
  \| e_n \| _{ d (p)} = n ^{1/p}, \quad n  \in \N ,
  \end{equation}
  \cite[Lemma 11(ii)]{ABR5}. Since $ d (p+) = \bigcap_{ k \in \N} d (p_k) $ for $ p \in [ 1, \infty )$ with $ p_k \downarrow p $
  and $ d (p\mbox{-}) = \mbox{ind}_k d (p_k) $ for $ 1 < p \leq  \infty $ with $ 1 < p_k \uparrow p,$ it is routine to check using
  \eqref{4.9} and the nature of the bounded subsets in the spaces $ d (p+), d (p\mbox{-})$ that $ \cE $
  is an \textit{unbounded}\/ subset in every such space $ d (p+) $  and $ d (p\mbox{-}). $ In particular, $ \{ e_n : n \in\N \}  $
  \textit{cannot}\/ be a  convergent sequence in any of  these spaces. 
     }

\end{Rem}

\begin{Prop} \label{P4.9} \begin{itemize} \item[{\rm (i)}]
\begin{itemize} \item[{\rm (a)}] For each $ 1 \leq p \leq q < \infty $ both  of the Ces\`{a}ro operators
$ C :  d(p+) \lra d (q+) $ and $ C : \ces (p+) \lra d (q+) $ are continuous.

\item[{\rm (b)}] Let $ 1 \leq p < \infty $ and $ x \in \C ^{\N} .$ Then it is the case that
$$
C^2 ( | x | ) \in d (p+)  \ \mbox{ if \ and \ only \ if } \ C (| x |) \in d (p+).
$$

\item[{\rm (c)}] For each $ 1 \leq p < \infty $ the identity $ [C, d (p+)]_s = \ces ( p+) $ is valid.
\end{itemize}

\item[{\rm (ii)}] \begin{itemize}\item[{\rm (a)}]  For each $ 1 < p \leq q  \leq  \infty $ both of the Ces\`{a}ro operators
$ C : d (p\mbox{-}) \lra d (q\mbox{-}) $ and $ C  : \ces (p\mbox{-}) \lra d (p\mbox{-}) $ are continuous.

\item[{\rm (b)}] Let $ 1 < p \leq \infty $ and $ x \in \C^{\N}.$ Then it is the case that
$$
C^2 ( | x | ) \in d (p\mbox{-}) \ \mbox{ if \ and \ only \ if }  \ C (| x | ) \in d (p\mbox{-}).
$$

\item[{\rm (c)}] For each $ 1 < p \leq \infty $ the identity $ [C, d (p\mbox{-})]_s = \ces (p\mbox{-})$  is valid.
\end{itemize}
\end{itemize}
\end{Prop}

\begin{proof} \ (i) \ (a) \ The continuity of $ C : d (p+) \lra d (q+) $ follows from Lemma 2.5(i) of \cite{ABR5},
the definition of $ d (p+) $ and $ d (q+), $ and the  fact that $ C : d (r) \lra d (s)$ is continuous whenever
$ 1 < r \leq s < \infty ;$ see \cite[Proposition 5.3(iii)]{BR1}).

A similar argument applies to establish the continuity of $ C : \ces (p+) \lra d (q+),  $ where now it is needed that $ C : \ces (r) \lra d (s) $
is continuous whenever $ 1 < r \leq s < \infty $ (cf.\ \cite[Proposition 5.3(v)]{BR1}).

(b) \ The proof of part (b) in Proposition \ref{P3.3}(i) can easily be adapted to apply to this case (by using part (a) above and \eqref{2.4} there
in place of \eqref{2.3}).

(c) \ Clearly $ \ces (p+) \subseteq [C, d (p+)]_s $ as $ C $ maps $ \ces (p+)$ into $ d (p+)$ by part (a). On the other hand, let $ X \subseteq \C^{\N} $
be a solid space such that $ C (X) \subseteq d (p+).$ Given $ x \in X $ also $ | x | \in X $ and hence, $ C (| x | ) \in d (p+) \subseteq \ces (p+).$
By Proposition \ref{P2.2} of \cite{ABR1} we have that $ x \in \ces (p+).$  Accordingly, $ X \subseteq \ces (p+).$
This implies that $ [C, d (p+)]_s \subseteq \ces (p+). $

(ii) \ (a) \ The proof of part (a) in (i) above can be adapted to apply here by replacing  Lemma 2.5(i) of \cite{ABR5} used there  with
Lemma 17(i) of \cite{ABR2}.

(b) \ The proof of part (iv) of Proposition \ref{P3.4} can be modified by using part (a) above, the definition $ d (p\mbox{-}) = \bigcup_{ 1 < q < p}  d(q) $ and applying
\eqref{2.4} in place of \eqref{2.3}.

 (c) \ Using part (a) above, the proof of part (v) of Proposition \ref{P3.4} can be adapted to fit the present setting.
 \end{proof}

\begin{Rem} \label{R4.10} {\rm (i) \  Since $ C : \ell_r \lra d (s) $  is continuous  whenever $ 1 < r \leq s < \infty , $
\cite[Proposition 5.3(iv)]{BR1}, it follows from Lemma 2.5(i) of \cite{ABR5}, resp. Lemma 17(i) of \cite{ABR2}, that $ C : \ell_{p+} \lra d (q+) $
is continuous whenever $ 1 \leq p \leq q < \infty ,$ resp. $ C : \ell_{p_{\mbox{-}}} \lra d (q\mbox{-})  $  is continuous   whenever $ 1 < p \leq q  \leq  \infty .$

(ii) \ The analogue of the stronger version of the ``Bennett property'' for $ C $ acting in $ \ces(p), 1 < p < \infty ,$ as it is  stated in
\eqref{2.1}, is known to also hold for $ C $ acting in $ \ces (p+),$  \cite[Proposition 2.2]{ABR1}, and for $ C $ acting in $ \ces (p\mbox{-}), 1 < p \leq \infty , $
\cite[Proposition 1(i)]{ABR2}. However, it \textit{fails}\/ to hold for $ C $ acting in $ \ell_{p+}, d(p+)$ and in $ \ell_{p_{\mbox{-}}}, d(p\mbox{-}).  $

Indeed, for $ \ell_{p+} $ with $ 1 \leq p < \infty ,$ see \cite[Proposition 2.4]{ABR1}. For $ \ell_{p_{\mbox{-}}} $ with $ 1 < p \leq \infty ,$ choose $ 0 \leq x \in \ces (p\mbox{-})
\backslash \ell_{p_{\mbox{-}}} $ (possible as the containment \eqref{3.10} is proper when  $ q = p$) and note that $ C ( | x | ) = C (x) \in \ell_{p_{\mbox{-}}} $ by Proposition \ref{P3.4}(ii).
So, it does \textit{not}\/ follow from $ C ( | u | ) \in \ell_{p_{\mbox{-}}} $ that necessarily $ u \in \ell_{p_{\mbox{-}}} .$

Concerning $ d (p+) $ with $ 1 \leq p < \infty , $ it follows from \eqref{4.4} and the ensuing discussion that there exists $ 0 \leq x \in \ell_{p+} \bsl d(p+).$
Then Proposition \ref{P4.9}(i)(a) implies that $ C ( |x| ) = C (x) \in d (p+). $ So, $ C ( | u | ) \in d (p+) $ need \textit{not}\/ imply that $ u \in d (p+).$
Finally, for $ d (p\mbox{-}), 1 < p  \leq \infty,$ choose $ 0 \leq x \in \ell_{p_{\mbox{-}}} \bsl d (p\mbox{-}),$ which is possible via \eqref{4.6} and the ensuing
discussion, and note that $ C ( | x |) = C (x) \in d (p\mbox{-})$ by part (i) of this remark. That is, it does \textit{not}\/ follow from $ C ( | u | ) \in d (p\mbox{-})$
that necessarily $ u \in d (p\mbox{-}).$

For a more general version of the {\em Bennett property} for Banach spaces see \cite{CR2}. } \hfill $\Box$
\end{Rem}
We conclude with some comments about the spaces occuring in this paper when they are considered as \textit{locally solid, lc-Riesz spaces}. The
standard reference on this topic (for real spaces) is \cite{AB}; for complex spaces see \cite{Z}.

Let $ 1 \leq p < \infty .$ Then $ \ces (p+),$ resp.  $ \ell_{p+},$ is the complexification of the corresponding real Riesz space $ \ces_{\R}   (p+)
:= \{ x \in \ces (p+) : x =  ( x_n ) _n \in \R^{\N}\}, $ resp.  $ ( \ell_{p+})_{\R}  := \{ x \in \ell_{ p+} : x = (x_n) _n \in \R^{\N}  \} ,$ where
the order in the real spaces is defined coordinatewise. The Fr\'{e}chet lattices $ \ces (p+), \ell_{p+} $ are \textit{Dedekind complete}, that is, every subset
of $ \ces_{\R} (p+), ( \ell_{p+} ) _{\R} $ which is bounded from above in the order sense has a least upper bound. Moreover, being \textit{reflexive}, each of the
(separable) Fr\'{e}chet lattices $ \ces ( p+), \ell_{p+}, $ for $ p \in [1, \infty ), $ has a \textit{Lebesgue topology}, that is, if $ x ^{(\alpha)} \downarrow 0 $
is a decreasing net in the order of $ \ces (p+), \ell_{p+} ,$ then $ \lim_{\alpha} x ^{ (\alpha)} = 0$ in the topology of $ \ces (p+), \ell_{ p+}.$
In addition, the order intervals in $ \ces (p+), \ell_{ p+} $ are topologically complete. Each Fr\'{e}chet lattice $ \ces (p+), $ for $ p \in [1, \infty ), $ is \textit{Montel},
which is \textit{not}\/ the case for $ \ell_{ p+}, $ for $  p \in [ 1, \infty ).$ For these notions and facts (and additional properties) we refer to
\cite[Section 4]{ABR1}. All of the properties needed for establishing the above facts for $ \ces (p+) $ in \cite{ABR1} are also available for $ d (p+). $ So, each space
$ d (p+), $ for $  p \in [ 1, \infty ),$ is a  (Montel) locally convex  Fr\'{e}chet lattice which is Dedekind complete, has a Lebesgue topology and its order intervals are
topologically complete.

Various properties of the associated (LB)-spaces $ \ell_{ p_{\mbox{-}}} , \ces (p\mbox{-}), $ for $ 1 < p \leq \infty ,$ considered as locally solid,  lc-Riesz spaces, occur in
\cite[Section 6]{ABR2}. The (LB)-spaces $ \ell_{p_{\mbox{-}}} $ are reflexive (but,  not Montel), Dedekind complete, have a Lebesgue topology and their  order intervals are topologically
complete. Each space $ \ces (p\mbox{-}), 1 < p \leq \infty ,$ has the same properties just listed for $ \ell_{p_{\mbox{-}}} $ and, in addition, is Montel. All of the properties needed
for establishing the above facts in \cite{ABR2} for $ \ces (p\mbox{-})$ are also available for $ d (p\mbox{-}).$ Hence,
each space $ d (p\mbox{-}), 1 < p \leq \infty ,$ is a Montel, locally solid,
lc-Riesz space which is Dedekind complete, has a Lebesgue topology and its order intervals are topologically complete. \\

\textbf{Acknowledegment.} \ The research of J.\ Bonet was partially supported by the projects MTM2016-76647-P and GV Prometeo/2017/102 (Spain).

\end{document}